\documentclass[12pt]{amsart}
\usepackage{amssymb,latexsym,amsmath,amscd,mathrsfs,yfonts,hyperref}
\usepackage{fullpage}
\usepackage{eucal}
\input xy
\xyoption{all}

\newcommand{\ZZ}{\mathbb Z}

\newcommand{\CC}{\mathbb C}

\newcommand{\NN}{\mathbb N}

\newcommand{\cpt}{\mathbb K}

\newcommand{\adj}[4]{#1\negmedspace: #2\rightleftarrows #3:\negmedspace #4}

\def\Pro{\textup{Pro}}
\def\coker{\textup{coker}}

\def\C{\textup{C}}

\def\ev{\textup{ev}}

\def\id{\mathrm{id}}

\def\Hom{\textup{Hom}}

\def\ker{\textup{ker}}

\def\op{\textup{op}}

\def\prot{\hat{\otimes}}

\def\Fun{\textup{Fun}}

\def\st{{\sf stab}}
\def\St{{\sf Stab}}

\def\NSH{\mathtt{NSH}}

\def\K{\textup{K}}
\def\SS{\mathbb{S}}

\def\S{\mathcal{S}}

\def\iNSp{\mathtt{NSp}}
\def\iNS{\mathtt{N}\mathcal{S_*}}
\def\hNS{\mathtt{hN}\mathcal{S_*}}
\def\hNSp{\mathtt{hNSp}}
\def\bu{\textup{bu}}
\def\cSf{\mathcal{S}^{\textup{fin}}_*}
\def\Sp{\mathtt{Sp}}
\def\hSp{\mathtt{hSp}}
\def\LC{\mathtt{LC}}
\def\LW{\mathtt{LC_W}}

\def\SHo{\mathtt{\Sigma Ho^{C^*}}}
\def\KK{\textup{KK}}
\def\E{\textup{E}}

\def\cC{\mathcal C}

\def\1{\bf{1}}

\def\SWf{\mathtt{SW^f}}
\def\Csep{\mathtt{SC^*}}

\def\cU{\mathcal U}

\def\iCsep{\text{$\mathtt{SC_\infty^*}$}}
\def\ilim{\varprojlim}
\def\holim{\textup{holim\,}}
\def\hocolim{\textup{hocolim\,}}
\def\dlim{\varinjlim}
\def\cT{\mathcal T}

\def\cD{\mathcal D}

\def\cI{\mathcal I}

\def\cA{\mathcal{A}}
\def\cB{\mathcal{B}}

\newcommand{\map}{\rightarrow}
\newcommand{\functor}{\rightarrow}

\def\h{\mathtt{h}}
\def\Exc{\mathtt{Exc}}
\def\Ind{\textup{Ind}}
\def\hosc{\mathtt{HoSC^*}}

\def\SH{\mathtt{SH}}
\newcommand{\beq}{\begin{eqnarray}}
\newcommand{\beqn}{\begin{eqnarray*}}
\newcommand{\eeq}{\end{eqnarray}}
\newcommand{\eeqn}{\end{eqnarray*}}

\theoremstyle{definition}
\newtheorem{thm}{Theorem}[section]

\newtheorem{lem}[thm]{Lemma}
\newtheorem{prop}[thm]{Proposition}
\newtheorem{cor}[thm]{Corollary}
\newtheorem{ex}[thm]{Example}
\newtheorem{defn}[thm]{Definition}
\newtheorem{rem}[thm]{Remark}
\newtheorem*{Rem}{Remark}
\newtheorem{que}[thm]{Question}

\begin{document}

\title{Noncommutative stable homotopy and stable infinity categories}
\author{Snigdhayan Mahanta}
\email{s.mahanta@uni-muenster.de}
\address{Mathematical Institute, University of Muenster, Einsteinstrasse 62, 48149 Muenster, Germany.}
\thanks{This research was supported by the Deutsche Forschungsgemeinschaft (SFB 878), ERC through AdG 267079, and the Humboldt Professorship of M. Weiss.}
\maketitle

\begin{abstract}
The noncommutative stable homotopy category $\NSH$ is a triangulated category that is the universal receptacle for triangulated homology theories on separable $C^*$-algebras. We show that the triangulated category $\NSH$ is topological as defined by Schwede using the formalism of (stable) infinity categories. More precisely, we construct a stable presentable infinity category of noncommutative spectra and show that $\NSH^\op$ sits inside its homotopy category as a full triangulated subcategory, from which the above result can be deduced. We also introduce a presentable infinity category of noncommutative pointed spaces that subsumes $C^*$-algebras and define the noncommutative stable (co)homotopy groups of such noncommutative spaces generalizing earlier definitions for separable $C^*$-algebras. The triangulated homotopy category of noncommutative spectra admits (co)products and satisfies Brown representability. These properties enable us to analyse neatly the behaviour of the noncommutative stable (co)homotopy groups with respect to certain (co)limits. Along the way we obtain infinity categorical models for some well-known bivariant homology theories like $\KK$-theory, $\E$-theory, and connective $\E$-theory via suitable (co)localizations. The stable infinity category of noncommutative spectra can also be used to produce new examples of generalized (co)homology theories for noncommutative spaces.
\end{abstract}

\setcounter{tocdepth}{2}
\tableofcontents

\begin{center}
{\bf Introduction}
\end{center}

The most widely studied bivariant homology theory of separable $C^*$-algebras is Kasparov's $\KK$-theory \cite{KasKK1,KasKK2}. A variant of $\KK$-theory for separable $C^*$-algebras with better excision properties was developed in \cite{HigETh} by using categories of fractions. The same theory attained a different description via {\em asymptotic homomorphisms} through the work of Connes--Higson \cite{ConHig}, which eventually came to be known as bivariant $\E$-theory. In the construction of bivariant $\E$-theory one applies a stabilization by the compact operators in order to enforce $C^*$-stability. It was shown separately by Connes--Higson and D\u{a}d\u{a}rlat \cite{DadAsymHom} that one obtains an aperiodic bivariant homology theory on the category of separable $C^*$-algebras after infinite suspension if the stabilization by compact operators is left out. Furthermore, the associated univariant (co)homology theory recovers the stable (co)homotopy theory of a finite CW complex $X$ when it is applied to the commutative and unital $C^*$-algebra $\C(X)$. Therefore, this bivariant homology theory is called the {\em noncommutative stable homotopy theory}. Houghton-Larsen--Thomsen showed in \cite{HouTho} that this theory admits a universal characterization, which is useful for analysing its formal properties. Finally Thom showed in \cite{ThomThesis} that the category of separable $C^*$-algebras equipped with the bivariant stable homotopy groups admits a natural triangulated category structure and called it the {\em noncommutative stable homotopy category}. We denote this triangulated category by $\NSH$. It contains the Spanier--Whitehead category of finite spectra as a full triangulated subcategory. Strictly speaking, the Spanier--Whitehead category sits contravariantly inside $\NSH$; however, in view of Spanier--Whitehead duality we may ignore this issue. The bivariant $\E$-theory category appears as a Verdier quotient of the triangulated category $\NSH$. Therefore, noncommutative stable homotopy theory is a sharper invariant than bivariant $\K$-theory for nuclear separable $C^*$-algebras. However, the triangulated category $\NSH$ is plagued by certain shortcomings from the viewpoint of homotopy theory. Its deficiencies are similar to those of the finite stable homotopy category. One way to address this problem is to construct a stable model category and show that $\NSH$ is a full triangulated subcategory of its homotopy category. This was an open question in \cite{ThomThesis} and in the language of Schwede \cite{SchwedeTri} this problem can be stated as: Is $\NSH$ a {\em topological triangulated category}? This question is also important from the perspective of the global structure of $\NSH$ and questions of this nature can be traced back to \cite{RosNCT}. It was shown by Andersen--Grodal that one cannot construct a model structure on the category of $C^*$-algebras that models the standard homotopy category of $C^*$-algebras \cite{AndGro} (see also \cite{Uuye,JoaJoh}). It must be noted that the category of all $C^*$-algebras is closed under small limits and colimits. Hence there is a trivial model category structure on the category of all $C^*$-algebras; isomorphisms are defined to be weak equivalences and every $*$-homomorphism is set to be a fibration as well as a cofibration. Its homotopy category is clearly not the standard homotopy category of $C^*$-algebras.

We solve the above-mentioned problem using the modern technology of infinity categories. We construct a stable infinity category of noncommutative spectra that has several appealing features to serve as the {\em right framework} for stable homotopy theory in noncommutative topology. To this end we freely use of the work of Lurie \cite{LurToposBook,LurHigAlg} on (stable) infinity categories, which uses the quasicategory model of Joyal \cite{Joyal,Joyal2} based on an earlier seminal work by Boardman--Vogt \cite{BoaVog}. The formalism of (stable) infinity categories is one of the several possible frameworks for studying higher category theory. It offers a robust setup with an extensive selection of computational tools and structural results. In the setting of stable infinity categories a universal characterization of higher algebraic $\K$-theory has recently been obtained by Blumberg--Gepner--Tabuada \cite{BluGepTab}. It will play an important role in a follow up article \cite{MyNSHLoc}, where we are going to provide some further applications of the setup that we develop here. We briefly discuss some related constructions in the literature.

Similar questions have been addressed in the setting of Quillen model categories by Joachim--Johnson in \cite{JoaJoh} and by {\O}stv{\ae}r in \cite{Ost}. The homotopy category of the model category constructed by Joachim--Johnson is an {\em enlarged $\KK$-category}. It is plausible that this model category is related to a (co)localization of our stable infinity category of noncommutative spectra (see Remark \ref{EbuKK}). The stable homotopy category of {\O}stv{\ae}r is motivic in nature and differs from $\NSH$. Let us mention that one can also perform homotopy theory within the world of $C^*$-algebras in the setting of a {\em category of fibrant objects} following Brown \cite{RosNCT,SchTopMet2,Uuye}. There is also a model category of $C^*$-categories \cite{DellTab}, whose homotopy category is different from the ones that we consider. Staying at the level of triangulated (homotopy) categories we have at our disposal the suspension stable homotopy category $\SHo$ of all (possibly nonseparable) $C^*$-algebras constructed in \cite{CunMeyRos}.

Topological triangulated categories have good structural properties but for computational purposes it is useful to have an {\em algebraic model}. Unfortunately, $\NSH$ is not {\em algebraic} according to the definition of Keller \cite{KelDG}. Indeed, algebraicity passes to triangulated subcategories and $\NSH$ contains a nonalgebraic triangulated subcategory, viz., the opposite of the finite stable homotopy category. However, there are reasons to remain optimistic on the algebraization problem. Recently Bentmann showed in \cite{Bentmann} that after restricting one's attention to suitable triangulated subcategories of certain (co)localizations of $\NSH$, one might expect algebraic models since they have infinite $n$-order. The $n$-order of a triangulated category was defined by Schwede in \cite{SchAlgTri} and it is infinite for an algebraic triangulated category. Here is a glossary of our constructions for the benefit of the reader:

\begin{enumerate}
\item $\iCsep =$ infinity category of separable $C^*$-algebras.
\item $\iNS =$ infinity category of pointed noncommutative spaces.
\item $\Sp(\iCsep) =$ minimal stabilization of separable $C^*$-algebras.
\item $\iNSp =$ stable infinity category of noncommutative spectra.
\end{enumerate}

There is also a stabilization infinity functor $\Sigma^\infty_S: \iNS\functor\iNSp$ that can be regarded as the {\em suspension spectrum functor} in noncommutative topology. The opposite of the infinity category $\iCsep$ is our model for the infinity category of pointed compact metrizable noncommutative spaces. The infinity categories $\iCsep$ and $\iNS$ are canonically enriched over the infinity category of pointed spaces. The stable infinity categories $\Sp(\iCsep)$ and $\iNSp$ are both useful for the study of bivariant homology theories; however, $\iNSp$ has better formal properties by design. The infinity category $\iNS$ (resp. $\iNSp$) is presentable so that one can always extract a combinatorial simplicial model category (resp. combinatorial simplicial stable model category), whose underlying infinity category is $\iNS$ (resp. $\iNSp$). As a byproduct of this methodology we obtain stable infinity categorical models for $\E$-theory, connective $\E$-theory or $\bu$-theory, and $\KK$-theory denoted by $\mathtt{E}_\infty$, $\mathtt{bu}_\infty$, and $\mathtt{KK}_\infty$ respectively. The associated triangulated homotopy categories all turn out to be topological (see Remark \ref{EbuKK}). Recently Ivankov--Meyer also announced an independent construction of infinity categorical models for $\E$-theory and $\KK$-theory. Let us mention that some of our constructions described above can be generalized to a larger class of topological algebras like locally convex algebras (see Remark \ref{locallyConvex}). An outline of the article is as follows:

In section \ref{NSH} we briefly recall the construction of the triangulated noncommutative stable homotopy category $\NSH$ following \cite{ThomThesis}. We analyse some of its features and also construct its $p$-localization as a monoidal triangulated category for any prime $p$.  The triangulated category $\NSH$ is a (rough) counterpart of the triangulated category of finite spectra; there are also some deviations (see \cite{MyNGH}). Nevertheless, it seems reasonable to expect that it be contained fully faithfully in the homotopy category of every model for noncommutative stable homotopy theory and this is the guiding philosophy behind the constructions in this article.

In Section \ref{Infty} we construct infinity categorical models for both unstable (denoted by $\iNS$) and stable (denoted by $\iNSp$) homotopy categories of noncommutative pointed spaces. We also construct a canonical fully faithful exact functor from the triangulated category $\NSH$ to the homotopy category of the stable infinity category $\iNSp^\op$. Using this result we give an affirmative answer to the question whether $\NSH$ is topological (see Theorem \ref{TopTri}). As a consequence we obtain two different descriptions of the noncommutative stable homotopy category $\NSH$; one is a convenient setting for analysing the formal categorical properties, while the other comes in handy for explicit computations. The stable infinity category of noncommutative spectra $\iNSp$ has several attractive features like canonical enrichment over spectra and the ease of further localization, to mention only a few. It also has the advantage of satisfying Brown representability in complete generality (see Theorem \ref{BR}). It is well-known that very few (co)homology theories are (co)respresentable in $\NSH$. Using the stable infinity category $\iNSp$ one can construct new generalized (co)homology theories for $C^*$-algebras that will appear elsewhere. In subsection \ref{Ostvaer} we also exhibit a comparison functor from $\NSH$ to the noncommutative stable homotopy category of {\O}stv{\ae}r that should be studied further.

Finally in Section \ref{Htpy} we define the noncommutative stable (co)homotopy groups of all noncommutative pointed spaces using noncommutative spectra and study their behaviour under certain (co)limits (see Theorem \ref{hNSpCont}). Our definitions generalize the earlier ones for separable $C^*$-algebras (see \cite{ConBook,DadAsymHom}). We emphasize the simplicity of the arguments in this section, which is made possible by our construction of $\iNSp$. It was mentioned earlier that in \cite{CunMeyRos} the authors introduced a suspension stable homotopy category $\SHo$ for all (possibly nonseparable) $C^*$-algebras. However, the authors themselves stated that $\SHo$ may not be the right stabilization of all $C^*$-algebras. For instance, it is not clear whether the behaviour of (co)representable functors under (co)limits in $\SHo$ can be phrased as easily as in our construction (see Remarks \ref{CunMeyRos} and \ref{nonsep}). We conclude with some basic computations for finite group $C^*$-algebras. In a forthcoming article we shall compute the noncommutative stable cohomotopy groups of $ax+b$-semigroup $C^*$-algebras associated to number rings.

\vspace{2mm}
\noindent
{\bf Notation and conventions:} In this article an infinity category will always mean an $(\infty,1)$-category and we are going to write it as $\infty$-category following the literature. Unless otherwise stated, a $C^*$-algebra is assumed to be separable. A limit (resp. colimit) in the context of $\infty$-categories is assumed to be an $\infty$-limit (resp. $\infty$-colimit). Similarly, a localization in the context of $\infty$-categories will mean a (Bousfield) localization of $\infty$-categories. In order to keep the exposition concise we have refrained from discussing the rich theory or history of $\infty$-categories. Interested readers are encouraged to peruse \cite{BoaVog, LurToposBook, LurHigAlg, Joyal, Joyal2}, amongst others, for material pertinent to the quasicategory model that we have used here and to consult \cite{Bergner} for a comparison of the different formalisms.

\begin{Rem}
Since $\Csep^\op$ is the category of noncommutative pointed compact metrizable spaces, it is $\NSH^\op$ that deserves the title of noncommutative stable homotopy category. The triangulated category $\NSH$ should be regarded as the suspension stable homotopy category of separable $C^*$-algebras as in \cite{CunMeyRos}. Nevertheless, we stick to the established nomenclature \cite{ThomThesis}.
\end{Rem}

\noindent
{\bf Acknowledgements.} The author would like to thank I. Barnea, J. Cuntz, D. Enders, B. Jacelon, M. Joachim, M. Marcolli, T. Nikolaus, S. Schwede, K. Strung and M. Weiss for helpful discussions. The author is thankful to P. A. {\O}stv{\ae}r for bringing \cite{Ost} to our attention and explaining some of its contents. The author is indebted to D. Gepner for generously sharing his knowledge of $\infty$-categories and A. Thom for very helpful feedback. The author is extremely grateful to the anonymous referee for suggesting several improvements. The author gratefully acknowledges the support and hospitality of MFO Oberwolfach, ESI Vienna, and MPIM Bonn under various stages of development of this project.

\section{Noncommutative stable homotopy $\NSH$} \label{NSH}
In this section we recall some basic facts about the (noncommutative) stable homotopy category. Connes--Higson \cite{ConBook} and D\u{a}d\u{a}rlat \cite{DadAsymHom} independently showed that the asymptotic homotopy classes of asymptotic homomorphisms between separable $C^*$-algebras lead to a satisfactory notion of noncommutative (stable) homotopy theory after suspension stabilization. This construction was put in the context of triangulated categories by Thom in \cite{ThomThesis}.

In stable homotopy theory the Spanier--Whitehead category of finite spectra, denoted by $\SWf$, is a fundamental object of study. It is constructed by formally inverting the suspension functor in the homotopy category of finite pointed CW complexes. More precisely, its objects are pairs $(X,n)$, where $X$ is a finite pointed CW complex and $n\in\ZZ$. Let $S$ denote the reduced suspension functor. The morphisms in this category are defined as $\SWf((X,n),(X',n')):= {\dlim}_r [ S^{r+n}X,S^{r+n'}X']$, where the colimit is taken over the suspension maps. The category $\SWf$ is a triangulated category, where the distinguished triangles are those which are equivalent to mapping cone triangles.

One possible formulation of the Gel'fand--Na{\u{\i}}mark correspondence is that the category of pointed compact metrizable spaces with pointed continuous maps is contravariantly equivalent to the category of separable commutative $C^*$-algebras with $*$-homomorphisms via the functor $(X,x)\mapsto \C(X,x)$. Here $\C(X,x)$ denotes the $C^*$-algebra of continuous complex valued functions on $X$ vanishing at the basepoint $x\in X$. Let us denote by $\Csep$ the category of separable $C^*$-algebras and by $\hosc$ its homotopy category. Let $\hosc[\Sigma^{-1}]$ be the category, whose objects are pairs $(A,n)$, $A\in\Csep$ and $n\in\ZZ$, with morphisms defined as $$\hosc[\Sigma^{-1}]((A,n),(B,m)):={\dlim}_r [\Sigma^{r+n} A,\Sigma^{r+m} B].$$ Here $\Sigma(-)=\C_0((0,1),-)$ denotes the suspension functor in the category of $C^*$-algebras and $[-,-]$ denotes homotopy classes of $*$-homomorphisms. This is the most direct generalization of $\SWf$ to the noncommutative setting. It admits a canonical triangulated category structure similar to the Spanier--Whitehead construction described above (see, e.g., \cite{Ambrogio} for more details). There is an evident functor $\Csep\functor\hosc[\Sigma^{-1}]$, which sends $A\in\Csep$ to $(A,0)$ and any $*$-homomorphism to its suspension stable homotopy class. For any $*$-homomorphism $f:A\map B$ the pullback $\C(f)$ of the diagram $[A\overset{f}{\map}{B}\overset{\ev_0}{\leftarrow} B[0,1)]$ in $\Csep$ is called the {\em mapping cone of $f$}. Any surjective $*$-homomorphism $f:A\map B$ in $\Csep$ gives rise to a canonical $*$-homomorphism $\ker(f)\map\C(f)$. Localizing the triangulated category $\hosc[\Sigma^{-1}]$ along this class of morphisms produces a new triangulated category that is called the {\em noncommutative stable homotopy category} $\NSH$. There is an exact localization functor $\hosc[\Sigma^{-1}]\functor\NSH$, which gives rise to a canonical composite functor $\iota:\Csep\functor\hosc[\Sigma^{-1}]\functor\NSH$. 

The morphisms in the localized triangulated category $\NSH$ are in general described by some isomorphism classes of {\em roof diagrams}, which are quite cumbersome. There is an alternative description, where every morphism can be represented (up to asymptotic homotopy) by a $*$-homomorphism. It is known that \beq\label{PiDef}\NSH((A,n),(B,m))\cong {\dlim}_r [[\Sigma^{r+n} A,\Sigma^{r+m} B]],\eeq where $[[-,-]]$ denotes the asymptotic homotopy classes of asymptotic homomorphisms. Recall that $\cU A:=\C_b([0,\infty),A)/\C_0([0,\infty),A)$ is called the {\em asymptotic algebra} of $A$. An {\em asymptotic homomorphism} from $A$ to $B$ is simply a $*$-homomorphism $A\map\cU B$. Two asymptotic homomorphisms $\phi_1,\phi_2: A\map \cU B$ are said to be {\em asymptotically homotopic} if there is a $*$-homomorphism $H: A\map \cU (B[0,1])$ such that $\cU(\ev_0)\circ H = \phi_1$ and $\cU (\ev_1)\circ H = \phi_2$. 

\begin{rem} \label{AsymHom}
The asymptotic algebra of a separable $C^*$-algebra is almost never separable. We are not going to regard it as an object in $\NSH$; it merely plays a role in the definition of asymptotic homomorphisms. If two asymptotic homomorphisms are homotopic as $*$-homomorphisms, then they are also asymptotically homotopic; the converse usually does not hold. Thus there is always a canonical map $[A,B]\map [[A,B]]$ for all $A,B\in\Csep$.
\end{rem}

\noindent
In order to avoid notational clutter the objects of the form $(A,0)$ will henceforth be simply denoted by $A$. A diagram of the form $$\Sigma C \map A \map B \map C$$ in $\NSH$ is called a {\em distinguished triangle} if (up to suspension) it is equivalent to a mapping cone extension \cite{CunSka} $$\Sigma C' \map \C(f)\map B'\overset{f}{\map} C'.$$ We quote the following result from \cite{ThomThesis}.

\begin{thm}[Thom]
Equipped with the distinguished triangles as described above and the maximal $C^*$-tensor product $\prot$, $\NSH$ is a tensor triangulated category.
\end{thm}

\begin{rem} 
The author denoted the noncommutative stable homotopy category by $\mathrm{S}$ in \cite{ThomThesis}. Since there is a profusion of the letter `$\mathrm{S}$' appearing in different contexts in this article, we have decided to denote it by $\NSH$. We hope that this descriptive choice of notation will avoid confusion in the literature.
\end{rem} 

\noindent
Performing localizations of this category one obtains interesting bivariant homology theories on the category of separable $C^*$-algebras, which is a viewpoint that was advocated in \cite{ThomThesis}.

\begin{ex} \label{bu}
We give only two examples below; some other interesting possibilities are easily conceivable.
\begin{enumerate}
\item By localizing all the corner embeddings $A\map A\prot\cpt$ one obtains the bivariant Connes--Higson $\E$-theory \cite{ConHig}. It follows that $\NSH(A,B)\cong\E_0(A,B)$ in the category of stable $C^*$-algebras; in fact, the stability of $B$ suffices.

\item By localizing all the corner embeddings $A\map M_2(A)$ one obtains a connective version of bivariant $\E$-theory that was introduced in \cite{ThomThesis}. It is a noncommutative generalization of bivariant $\bu$-theory of spaces.
\end{enumerate}

\end{ex}

\begin{rem}
The Spanier--Whitehead category of spectra (not necessarily finite) is not the right framework for stable homotopy theory. For instance, it does not contain arbitrary coproducts. However, there is a consensus that every model for stable homotopy category should contain the finite Spanier--Whitehead category $\SWf$ as a full triangulated subcategory of its homotopy category. Since we are roughly dealing with the noncommutative version of the finite Spanier--Whitehead category, the na{\"i}ve stabilization described above suffices.
\end{rem}

\subsection{Localization at a prime $p$}
Experience from algebraic topology teaches us that it might be worthwhile to study the noncommutative stable homotopy category by localizing it at various primes. There is an easy construction of a $p$-local version of $\NSH$. For any prime number $p$ one can define the $p$-local version of $\NSH$, denoted by $\NSH_p$, by tensoring the $\Hom$-groups $\NSH(-,-)$ with $\ZZ_{(p)}$. Here $\ZZ_{(p)}$ denotes the localization of $\ZZ$ at the prime ideal $p\ZZ$. Let us set $\SS=(\CC,0)$ in $\NSH$, which is also the unit object with respect to the tensor product $\prot$. Then there is an exact localization functor $\NSH\functor\NSH_p$ between triangulated categories. Furthermore, $\NSH_p$ admits the structure of a tensor triangulated category, such that if we denote by $\SS_p$ the image of $\SS$ under the localization functor, then $\SS_p$ is a unit object in $\NSH_p$ and the localization functor is monoidal (see, for instance, Theorem 3.6 of \cite{BalSpectra}). For the benefit of the reader we record it as a Lemma.

\begin{lem} \label{pLoc}
There is a monoidal exact $p$-localization functor $\NSH\functor\NSH_p$.
\end{lem}

\noindent
Note that $\SWf$ is equivalent to its opposite category due to Spanier--Whitehead duality and it sits inside $\NSH$ via the construction $(X,x)\mapsto\C(X,x)$. Therefore, various results concerning $\SWf$ continue to hold inside $\NSH$. For any object $A\in\NSH$, the $n$-fold multiple of $\id_A\in\NSH(A,A)$ is denoted by $n\cdot A$ and the mapping cone of this morphism is denoted by $A/n$.

\begin{prop} \label{Prime}
For an odd prime $p$, one has $p\cdot A/p = 0$ for every $A\in\NSH$. 
\end{prop}

\begin{proof}
It is known that for an odd prime $p$, one has $p\cdot\SS_p = 0$ in $\SWf$ (see Proposition 5 of \cite{SchwedeTri}). In $\NSH$ there is an exact triangle $$\Sigma\SS \map \SS/p \map \SS\overset{p\cdot\SS}{\map} \SS.$$ Applying $A\prot -$ to the above triangle one obtains the distinguished triangle $$\Sigma A\map A/p\map A\overset{p\cdot A}{\map} A.$$ Now $p\cdot A/p \simeq p \cdot (\SS_p\prot A) \simeq (p\cdot \SS_p)\prot \id_A \simeq 0$. 
\end{proof}

\begin{rem}
The author is grateful to the referee for pointing out that the na{\"i}ve localization procedure described above may not be the best option from the viewpoint of Bousfield localization. In the next section we are going to embed $\NSH^\op$ inside a compactly generated triangulated category $\hNSp$ (see Theorem \ref{ExInc}), which facilitates such constructions. 
\end{rem}

\section{Noncommutative spaces and noncommutative spectra} \label{Infty}
We are interested in (co)localizations of the noncommutative stable homotopy category $\NSH$ since one can produce generalized (co)homology theories for separable $C^*$-algebras via this procedure. One would have access to a very elegant theory of (co)localization (enabling constructions like Lemma \ref{pLoc} in a generalized setup) if one could construct a combinatorial stable model category, whose homotopy category contained $\NSH$ as a full triangulated subcategory. Various features of $\NSH$ (including Proposition \ref{Prime} above) indicate that such a construction is viable demonstrating that $\NSH$ has a homotopy theoretic origin. We rephrase this as a natural question:

\begin{que}
Is $\NSH$ a topological triangulated category? 
\end{que} 

\noindent
Roughly speaking, a {\em topological triangulated category} is one which is equivalent to a full triangulated subcategory of the homotopy category of a stable model category (see \cite{SchTopTri} for a more accurate definition). Since $\NSH$ does not contain arbitrary coproducts, it cannot be equivalent to the homotopy category of a stable model category. One can try to rectify the situation by passing to the Ind-completion of $\NSH$. However, it is known that the Ind-completion of a triangulated category need not be triangulated \cite{ChrStr}. Moreover, several important constructions in homotopy theory rely on manoeuvres in the actual category of spectra, rather than its homotopy category. Thus we construct a stable $\infty$-categorical model of noncommutative spectra and show that $\NSH$ naturally sits (contravariantly) inside its homotopy category as a fully faithful triangulated subcategory. Our constructions rely on the elegant framework of (stable) $\infty$-categories developed by Joyal and Lurie \cite{Joyal,Joyal2,LurToposBook,LurHigAlg}. From the stable $\infty$-category of noncommutative spectra one also gets a combinatorial simplicial stable model category, whose homotopy category houses $\NSH^\op$ explaining its homotopy theoretic origin.

\subsection{An $\infty$-category of pointed compact metrizable noncommutative spaces}
The category of $C^*$-algebras is canonically enriched over that of pointed topological spaces. In order to remember the higher homotopy information it is important to keep track of the topology on the mapping sets. For any pair of $C^*$-algebras $A,B$, we equip the set of $*$-homomorphisms, denoted by $\Hom(A,B)$, with the topology of pointwise norm convergence. In the category of separable $C^*$-algebras $\Hom(A,B)=\Csep(A,B)$ is a metrizable topological space. Indeed, fix a sequence $\{a_n\}_{n\in\NN}$ in $A$, such that $\textup{lim}\, a_n = 0$ and the $\CC$-linear span of the set $\{a_n\}$ is dense in $A$. Then the metric $d(f_1,f_2) = \textup{sup}\{\| f_1(a_n)-f_2(a_n)\|_B\,|\, n\in\NN\}$ defines the desired topology on $\Csep(A,B)$. In fact, the category of separable $C^*$-algebras $\Csep$ is enriched over the category of pointed metrizable topological spaces (see Proposition 23 of \cite{Mey1}). We are going to refer to $\Csep$ as a {\em topological category}, when we endow it with the aforementioned enrichment but we do not introduce a new notation for it.

\begin{defn}
By taking the topological nerve of the topological category $\Csep$ (as in Section 1.1.5 of \cite{LurToposBook}) we obtain an $\infty$-category. We denote this $\infty$-category by $\iCsep$ and it is the {\em $\infty$-category of separable $C^*$-algebras}. Its opposite $\infty$-category is the {\em $\infty$-category of pointed compact metrizable noncommutatives spaces}.
\end{defn} 

\begin{rem}
The topological nerve of the topological category of CW complexes, where the mapping spaces are equipped with the compact-open topology, is called the {\em $\infty$-category of spaces} and denoted by $\S$. The $\infty$-category $\S$ plays a distinguished role since every $\infty$-category $\cC$ is canonically enriched over $\S$, i.e., for any $x,y\in\cC$ the mapping space $\cC(x,y)\in\S$. In particular, $\iCsep$ is enriched over $\S$.
\end{rem}

\begin{rem} \label{size}
Since the objects of $\Csep$ are separable $C^*$-algebras, it admits a small skeleton. For the sake of definiteness one could select those separable $C^*$-algebras that are concretely represented as $C^*$-subalgebras of $B(H)$ for a fixed separable Hilbert space $H$. One may replace $\iCsep^\op$ by this equivalent small $\infty$-category in order to circumvent potential set-theoretic issues in the sequel. 
\end{rem}

\subsection{A minimal stabilization of $\iCsep$}
The natural domain for studying stable phenomena in the setting of $\infty$-categories is that of {\em stable $\infty$-categories} \cite{LurHigAlg}. Rather tersely, it can be described as an $\infty$-category with a zero object $0$, such that every morphism admits a fiber and a cofiber, and the {\em fiber sequences} coincide with the {\em cofiber sequences}. Recall that a {\em fiber sequence} (resp. a {\em cofiber sequence}) is a pullback (resp. a pushout) square 
\beqn
\xymatrix{
X\ar[r]^f \ar[d] & Y\ar[d]^g\\
0\ar[r] & Z.
}
\eeqn and in a stable $\infty$-category the two notions are equivalent. An $\infty$-functor between two stable $\infty$-categories is called {\em exact} if it preserves all finite limits or, equivalently, if it preserves all finite colimits. Every pointed $\infty$-category $\cC$ admitting finite limits has a {\em loop functor} $\Omega_\cC:\cC\functor\cC$ defined as a pullback (up to a contractible space of choices)

\beqn
\xymatrix{
Y\simeq \Omega_\cC X\ar[r]\ar[d] & 0'\ar[d]\\
0\ar[r] & X,
}
\eeqn where $0,0'$ are zero objects. The dual construction, denoted by $\Sigma_\cC$,  produces the {\em suspension functor}, provided $\cC$ admits finite colimits. In a finitely bicomplete and pointed $\infty$-category $(\Sigma_\cC,\Omega_\cC):\cC\functor\cC$ form an adjoint pair. The aim of stabilization is to invert the functor $\Omega_\cC$.

\begin{rem}
Given any $C^*$-algebra one can construct its suspension staying within the category of $C^*$-algebras. However, in order to construct its homotopy adjoint {\em loop algebra} one must leave the world of $C^*$-algebras \cite{RosNCT}. One needs the full strength of pro $C^*$-algebras \cite{PhiHtpy}. We are going to address this issue in a slightly different manner (see Definition \ref{NCspaces} and Remark \ref{compGen} below).
\end{rem}

Let $\cA$ be a pointed $\infty$-category with finite colimits and let $\cB$ be an $\infty$-category with finite limits. Then an $\infty$-functor $F:\cA\functor\cB$ is called {\em excisive} if it sends a pushout square in $\cA$ to a pullback square in $\cB$ and it is called {\em reduced} if $F(\ast)$ is a final object in $\cB$, where $\ast\in\cA$ is a zero object. Let $\Exc_*(\cA,\cB)$ denote the full $\infty$-subcategory of the $\infty$-functor category $\mathtt{Fun}(\cA,\cB)$ spanned by the reduced excisive $\infty$-functors $\cA\functor\cB$. The $\infty$-category $\Exc_*(\cA,\cB)$ is stable (see Proposition 1.4.2.16. of \cite{LurHigAlg}). Let $\cSf$ denote the $\infty$-category of finite pointed spaces, which is 
an $\infty$-subcategory of the $\infty$-category of pointed spaces $\S_*$.

\begin{ex} \label{EnrichSp}
The stable $\infty$-category $\Exc_*(\cSf,\cB)$ is usually denoted by $\Sp(\cB)$ and its objects are called the {\em spectrum objects of $\cB$}. Setting $\cB=\S_*$ produces Lurie's model for the stable $\infty$-category of spectra, which is simply denoted by $\Sp$. Moreover, any stable $\infty$-category is canonically enriched over $\Sp$ \cite{LurHigAlg} (see also Example 7.4.14. of \cite{GepHau}).
\end{ex}

\noindent
We are going to stabilize $\iCsep$ using the procedure described above. To this end we show:

\begin{prop} \label{finlim}
The $\infty$-category $\iCsep$ possesses finite limits.
\end{prop}

\begin{proof}
Since the $\infty$-category $\iCsep$ is constructed by taking the topological nerve of $\Csep$, it suffices to show that the topological category $\Csep$ admits finite homotopy limits (see Remark 1.2.13.6. of \cite{LurToposBook}). In fact, by dualizing Corollary 4.4.2.4. of \cite{LurToposBook} it suffices to check that $\Csep$ admits homotopy pullbacks and possesses a final object (which it evidently does). 

The category $\Csep$ actually admits all finite (ordinary) limits. Indeed, it is known that every small limit exists in the category of all (possibly nonseparable) $C^*$-algebras (see Proposition 19 of \cite{Mey1}). It can be proven by showing the existence of small products and equalizers of pairs of parallel morphisms. For an arbitrary set of $C^*$-algebras $\{A_i\}_{i\in I}$, one can easily construct a product $C^*$-algebra $\prod^{C^*}_{i\in I} A_i$ consisting of norm-bounded sequences of elements with the sup norm; the equalizer of any pair of parallel morphisms $f_1,f_2:A{\rightrightarrows}B$ is given by $$\ker(f_1-f_2)=\{a\in A\,|\, f_1(a)=f_2(a)\}\subseteq A.$$ The explicit description of the (ordinary) product reveals that the category of separable $C^*$-algebras admits all finite (ordinary) products. It is also clear that the equalizer of any pair of parallel morphisms in $\Csep$ exists within it, whence $\Csep$ actually admits all (ordinary) finite limits. Now finite homotopy pullbacks can be constructed using standard techniques (see Chapter 11 of \cite{BouKan}); note that for any finite pointed simplicial set $(K,k)$ and any $A\in\Csep$ the natural {\em function object} $\C((|K|,k),A)\cong\C(|K|,k)\prot A$ exists in $\Csep$.
\end{proof}

\begin{rem}
 The $\infty$-category $\iCsep$ is also pointed (there is a zero $C^*$-algebra).
\end{rem}

\begin{lem} \label{loop}
In the $\infty$-category $\iCsep$ one has $\Omega_{\iCsep} A \cong\Sigma A = \C_0((0,1),A)$. 
\end{lem}

\begin{proof}
For any $C^*$-algebra $A$, the homotopy pullback $\Omega_{\iCsep} A$ of $0\map A\leftarrow 0'$ in the topological category $\Csep$ is characterized by a weak equivalence $$\Csep(D,\Omega_{\iCsep} A)\simeq \holim [ \Csep(D,0)\map\Csep(D,A)\leftarrow\Csep(D,0')]$$ for every $D\in\Csep$. Here the weak equivalence is in the category of pointed topological spaces over which $\Csep$ is enriched. The suspension--cone short exact sequence of $C^*$-algebras $$0\map \Sigma A=\C_0((0,1),A) \map\C_0([0,1),A)\map A\map 0$$ can also be viewed as a pullback diagram in $\Csep$ 

\beqn
\xymatrix{
\Sigma A\ar[r] \ar[d] & \C_0([0,1),A)\ar[d]\\
0\ar[r] & A,
} 
\eeqn where the right vertical arrow is a Schochet fibration (or a cofibration according to the terminology of \cite{SchTopMet2}). Now observe that $\C_0([0,1),A)$ is homotopy equivalent to $0$ in $\Csep$, exhibiting $\Sigma A$ as a homotopy pullback in the topological category $\Csep$. Indeed, for every $D\in\Csep$ one has \beqn&&\holim [ \Csep(D,0)\map\Csep(D,A)\leftarrow\Csep(D,\C_0([0,1),A))] \\ &\simeq&\holim  [\ast \map \Csep(D,A)\leftarrow \textup{P}\Csep(D,A)]\\ &\simeq&\Omega\Csep(D,A)\eeqn in the category of pointed topological spaces. Here we have used Proposition 24 of \cite{Mey1} in order to identify $\Csep(D,\C_0([0,1),A))\cong \textup{P}\Csep(D,A)\simeq \ast$. Observe that the canonical map $\textup{P}\Csep(D,A)\map\Csep(D,A)$ is a fibration. Finally using Proposition 24 of \cite{Mey1} once again we conclude that $\Omega \Csep(D,A)\cong\Csep(D,\Sigma A)$. 
\end{proof}

\begin{lem} \label{mapcone}
For any morphism $f:A\map B$ in $\Csep$ the mapping cone construction 
\beqn
\xymatrix{
\C(f)\ar[r]\ar[d] & \C_0([0,1),B)\ar[d]\\
A \ar[r]^f & B
}
\eeqn
is a homotopy pullback in the topological category $\Csep$. 
\end{lem}

\begin{proof}
 The proof is similar to that of Lemma \ref{loop} using the fact that the functor $\Csep(D,-):\Csep\functor\mathtt{Top_*}$ preserves pullbacks for all $D\in\Csep$ (see Corollary 2.6 of \cite{Uuye}). 
\end{proof}

\begin{rem}
The distinguished triangles in $\NSH$ can also be written as $$\Omega_{\iCsep}C\map A\map B\map C,$$ which is more in keeping with the conventions in topology.
\end{rem}

It follows that the above stabilization scheme is applicable to $\iCsep$. Following standard practice we are going to denote the homotopy category of any $\infty$-category $\cA$ by $\mathtt{h}\cA$. It is known that the homotopy category of any stable $\infty$-category is triangulated (see Theorem 1.1.2.14 of \cite{LurHigAlg}). The distinguished triangles are induced by the (co)fiber sequences described above. This phenomenon is one of the delightful features of stable $\infty$-categories; the simple and intuitive definition of stable $\infty$-categories (expressed as a property) produces quite elegantly triangulated categories as their homotopy categories. For the benefit of the reader we record the essential features of $\Sp(\iCsep)$ below.

\begin{prop}
The stable $\infty$-category $\Sp(\iCsep)$ is canonically enriched over $\Sp$ and the homotopy category $\h\Sp(\iCsep)$ is triangulated.
\end{prop}

This {\em loop stable} triangulated category is helpful to construct new generalized (co)homology theories on the category of separable $C^*$-algebras. However, we would like to model noncommutative stable homotopy and the above procedure produces a different stabilization. Moreover, this category does not admit all small colimits, which is a desirable feature. The second problem could have been rectified by enlarging $\Sp(\iCsep)$ directly by formally adjoining all infinite colimits. We are going to follow a different route as it also produces an $\infty$-categorical model for pointed noncommutative spaces that are not necessarily compact.

\subsection{An $\infty$-category of pointed noncommutative spaces}
For any regular cardinal $\kappa$ there is a formal procedure to adjoin $\kappa$-filtered colimits beginning with an $\infty$-category $\cA$. The construction is denoted by $\Ind_\kappa(\cA)$ and it is characterized by the property that it admits $\kappa$-filtered colimits and there is a Yoneda $\infty$-functor $j:\cA\functor\Ind_\kappa(\cA)$ that induces an equivalence of $\infty$-categories $$\Fun_\kappa(\Ind_\kappa(\cA),\cB)\functor\Fun(\cA,\cB),$$ for any $\infty$-category $\cB$, which admits $\kappa$-filtered colimits. Here $\Fun(-,-)$ [resp. $\Fun_k(-,-)$] denotes the $\infty$-category of $\infty$-functors [resp. the $\infty$-category of $\kappa$-continuous $\infty$-functors]. For further details see Section 5.3.5 of \cite{LurToposBook}. Let $\iCsep^\op$ denote the $\infty$-category, which is the opposite of $\iCsep$, and let $\iNS=\Ind_\omega(\iCsep^\op)$.

\begin{defn} \label{NCspaces}
We define the $\infty$-category $\iNS=\Ind_\omega(\iCsep^\op)\simeq \Pro_\omega(\iCsep)^\op$ to be the $\infty$-category of {\em pointed noncommutative spaces} that are not necessarily compact.  
\end{defn}

\noindent
One useful property that an $\infty$-category may possess is {\em presentability}. Intuitively, the presentability of an $\infty$-category ensures that it is a widely accommodating category (admits small colimits) and yet it is built on a small amount of data by simple procedures (accessible). The $\infty$-category $\iCsep^\op$ admits all finite colimits; the assertion follows from the dual of Proposition \ref{finlim}. Now Theorem 5.5.1.1 (4) of \cite{LurToposBook} shows that

\begin{lem}
The $\infty$-category $\iNS=\Ind_\omega(\iCsep^\op)$ is presentable. 
\end{lem} The canonical Yoneda embedding $j:\iCsep^\op\map \Ind_\omega(\iCsep^\op)=\iNS$ preserves all finite colimits that exist in $\iCsep^\op$ (see Proposition 5.3.5.14 of \cite{LurToposBook}). It follows from Lemma \ref{loop} that $\Sigma_{\iNS} j(A) \cong \Omega_{\iCsep}A$ for all objects $A\in\iCsep^\op$. The $\infty$-category $\iNS$ admits all finite limits (in fact, all small limits) and it is pointed by $j(0)$. In the sequel we are going to view separable $C^*$-algebras as objects of $\iNS$ via the Yoneda embedding $j$ but suppress it from the notation for brevity. Applying the above stabilization mechanism we obtain a stable $\infty$-category $\Sp(\iNS)$. However, $\Sp(\iNS)$ is not our end goal; it is merely a precursor. 

It follows from Proposition 1.4.4.4. of \cite{LurHigAlg} that $\Sp(\iNS)$ is a presentable stable $\infty$-category equipped with a canonical stabilization $\infty$-functor $\Sigma^\infty:\iNS\functor\Sp(\iNS)$. The composition of $\Sigma^\infty$ with the Yoneda embedding $j:\iCsep^\op\functor \iNS$ gives rise to an $\infty$-functor $\St:=\Sigma^\infty\circ j:\iCsep^\op\functor\Sp(\iNS)$. The $\infty$-functor $\St$ should be regarded as the suspension stabilization of a separable $C^*$-algebra. There is also an opposite functor $\Pi:=\St^\op:\h\iCsep\functor\hSp(\iNS)^\op$ at the level of homotopy categories. Note that $\h\iCsep\simeq\hosc$ is the homotopy category of separable $C^*$-algebras. Due to the symmetry in the definition of a stable $\infty$-category, the opposite of a stable $\infty$-category is also stable. Consequently the $\infty$-category $\Sp(\iNS)^\op$ is stable and its homotopy category $\h\Sp(\iNS)^\op$ is canonically triangulated.

\begin{rem} \label{compGen}
The image of $j:\iCsep^\op\map\iNS = \Ind_\omega(\iCsep^\op)$ provides a set of {\em compact objects} in $\iNS$, which generates $\iNS$ under filtered colimits (see Proposition 5.3.5.5 and Proposition 5.3.5.11 of \cite{LurToposBook}). In fact, by construction $\iNS$ is {\em compactly generated} (see Section 5.5.7 of \cite{LurToposBook}). Due to the presentability of $\iNS$ there is a combinatorial simplicial model category, whose underlying $\infty$-category is equivalent to $\iNS$ (see Proposition A.3.7.6 of \cite{LurToposBook}).
\end{rem}

\begin{rem} \label{Susp}
Using Proposition 1.4.2.24. of \cite{LurHigAlg} we conclude that $\Sp(\iNS)$ is equivalent to the inverse limit of the tower of $\infty$-categories $$\cdots \map \iNS \overset{\Omega_{\iNS}}{\map}\iNS \overset{\Omega_{\iNS}}{\map}\iNS$$ in the $\infty$-category of $\infty$-categories. Hence we recover the familiar description of $\Omega$-spectrum like objects from which one may conclude
\beqn
\hSp(\iNS)^\op(\Pi(A),\Pi(B))\cong\hSp(\iNS)(\St(B),\St(A))&\cong&{\dlim}_r \hNS(\Sigma_{\iNS}^r B,\Sigma_{\iNS}^r A) \\
&\cong&{\dlim}_r \h\iCsep(\Omega_{\iCsep}^r A,\Omega_{\iCsep}^r B)\\
&\cong& {\dlim}_r \hosc(\Sigma^r A,\Sigma^r B)\\
&=& {\dlim}_r [\Sigma^r A,\Sigma^r B]\eeqn for all separable $C^*$-algebras $A,B$.
\end{rem}

\begin{prop} \label{Ess}
The functor $\Pi:\hosc\simeq\h\iCsep\functor\hSp(\iNS)^\op$ induces a fully faithful exact functor $\Pi:\hosc[\Sigma^{-1}]\functor\hSp(\iNS)^\op$ between triangulated categories.
\end{prop}

\begin{proof}
Since the opposite of a stable $\infty$-category is also stable, the homotopy category $\hNSp^\op$ is triangulated. The functor $\Pi:\hosc\simeq\h\iCsep\functor\hSp(\iNS)^\op$ commutes with the $C^*$-suspension functor and hence descends to a unique functor $\Pi:\hosc[\Sigma^{-1}]\functor\hSp(\iNS)^\op$. By Lemma \ref{mapcone} every mapping cone diagram in $\Csep$ gives rise to a cofiber sequence in $\iCsep^\op$. Since $j$ preserves finite colimits, each such cofiber sequence becomes a cofiber sequence in $\iNS$ via $j$. The functor $\Sigma^\infty:\iNS\functor\Sp(\iNS)$ also preserves cofiber sequences. In $\Sp(\iNS)^\op$ this becomes a fiber sequence, which is also a cofiber sequence due to the stability of $\Sp(\iNS)^\op$. We conclude that the functor $\Pi:\hosc[\Sigma^{-1}]\functor\hSp(\iNS)^\op$ is exact. The functor $\Pi$ is determined on objects due to the commutativity constraint and compatibility with the (de)suspension functor. Recall from Section \ref{NSH} that $\hosc[\Sigma^{-1}]((A,n),(B,m))= {\dlim}_r [\Sigma^{r+n} A,\Sigma^{r+m} B]$ and the above Remark shows that $\hSp(\iNS)^\op(\Pi(A),\Pi(B))\cong{\dlim}_r [\Sigma^r A,\Sigma^r B]$. Since the suspension functor in $\hSp(\iNS)^\op$ is invertible it can be seen that $\hSp(\iNS)^\op(\Pi((A,n)),\Pi((B,m)))\cong{\dlim}_r [\Sigma^{r+n} A,\Sigma^{r+m} B]$. For every $(A,n),(B,m)\in\hosc[\Sigma^{-1}]$ there is a candidate for the functor $\Pi:\hosc\map\hSp(\iNS)^\op$ that induces a bijection between the morphism-groups $$\hosc[\Sigma^{-1}]((A,n),(B,m))\overset{\sim}{\map}\hSp(\iNS)^\op(\Pi((A,n)),\Pi((B,m)))$$ and satisfies the commutativity constraint. Due to the uniqueness of $\Pi$ it must be this one, which is manifestly a fully faithful functor. 
\end{proof}

\subsection{The stable $\infty$-category of noncommutative spectra $\iNSp$} \label{NSpLoc}
The stable $\infty$-category of noncommutative spectra is obtained as a (Bousfield) localization of $\Sp(\iNS)$. The goal of this localization $L:\Sp(\iNS)\functor S^{-1}\Sp(\iNS)$ is to achieve the following: all short exact sequences of separable $C^*$-algebras are made to behave like cofiber sequences in $\iCsep^\op$. 

For any $*$-homomorphism $f:B\map C$ in $\Csep$ there is a canonical map $\theta(f):\ker(f)\map\C(f)$ in $\iCsep$, which can be viewed as an element in $\iCsep^\op(\C(f),\ker(f))$. Now $\St(\theta(f))$ is a morphism in $\Sp(\iNS)$. We construct a {\em strongly saturated collection} of morphisms $S$ (see Definition 5.5.4.5 of \cite{LurToposBook}) in $\Sp(\iNS)$, which is of small generation starting from the small set $$S_0=\{\St(\theta(f))\,|\, \text{$f$ surjective $*$-homomorphism in $\Csep$}\}$$ that is compatible with the triangulation as follows: Let $\cA$ denote the stable $\infty$-subcategory of $\Sp(\iNS)$ generated by the set $\{\textup{cone}(g)\,|\, g\in S_0\}$. Then $\Ind_\omega(\cA)$ is a stable presentable $\infty$-subcategory of $\Sp(\iNS)$ (see Proposition 1.1.3.6 of \cite{LurHigAlg}). Let $S$ denote the class of maps in $\Sp(\iNS)$, whose cones lie in the essential image of $\Ind_\omega(\cA)$. We deduce from Proposition 5.6 of \cite{BluGepTab} that $S$ is a strongly saturated collection of small generation. Using the machinery of Section 5.5.4 of \cite{LurToposBook} we can construct an accessible localization $\infty$-functor $L:\Sp(\iNS)\functor S^{-1}\Sp(\iNS)$. It also follows from Proposition 5.6 of \cite{BluGepTab} that $S^{-1}\Sp(\iNS)\simeq \Sp(\iNS)/\Ind_\omega(\cA)$ as stable $\infty$-categories. Being born out of an accessible localization, the $\infty$-category $S^{-1}\Sp(\iNS)$ is presentable (see Remark 5.5.1.6 of \cite{LurToposBook}) and the localized $\infty$-category can be viewed as the full $\infty$-subcategory of $\Sp(\iNS)$ spanned by the {\em $S$-local} objects. Here an object $X\in\Sp(\iNS)$ is called {\em $S$-local} if, for each morphism $f:Y'\map Y$ in $S$, the composition with $f$ induces a homotopy equivalence $\Sp(\iNS)(Y,X)\overset{\sim}{\map}\Sp(\iNS)(Y',X)$. The localization $\infty$-functor $L$ is the left adjoint to the inclusion of the full $\infty$-subcategory of $S$-local objects inside $\Sp(\iNS)$. There is a composite right adjoint functor $S^{-1}\Sp(\iNS)\hookrightarrow\Sp(\iNS)\overset{\Omega^\infty_*}{\map} \iNS$, which plays the role of underlying infinite loop (noncommutative) space. Recall that an $\infty$-functor between stable $\infty$-categories is called {\em exact} if it commutes with finite limits (or, equivalently, if it commutes with finite colimits). We summarize the above discussion in the following:

\begin{lem} \label{pres}
The localization $L:\Sp(\iNS)\functor S^{-1}\Sp(\iNS)$ is an exact $\infty$-functor between stable and presentable $\infty$-categories.
\end{lem}

\noindent  
Motivated by Lurie's constructions and Thom's results we propose the following:

\begin{defn} \label{NSpectra}
We define the {\em stable $\infty$-category of noncommutative spectra} to be $$\iNSp:= S^{-1}\Sp(\iNS)$$ and the {\em triangulated homotopy category of noncommutative spectra} to be $$\hNSp:=\h S^{-1}\Sp(\iNS).$$ The composite $\infty$-functor $\Sigma^\infty_S :=L\circ\Sigma^\infty:\iNS\functor \iNSp =S^{-1}\Sp(\iNS)$ between presentable $\infty$-categories is the noncommutative analogue of the passage from pointed spaces to spectra.
\end{defn}

\begin{rem}
Being stable the $\infty$-category $\iNSp$ is enriched over $\Sp$ (see Example \ref{EnrichSp}). Besides the homotopy category of noncommutative spectra $\hNSp$ is canonically triangulated.
\end{rem}

\begin{rem} \label{Pi}
The construction above furnishes a canonical $\infty$-functor $\st:= L\circ\St :\iCsep^\op\functor\iNSp$. There is also an opposite functor $\pi:=\st^\op:\h\iCsep\functor\hNSp^\op$ at the level of homotopy categories. Thus the opposite of the triangulated homotopy category of noncommutative spectra bears a direct relationship with the homotopy category of separable $C^*$-algebras. This means that bivariant homology theories for separable $C^*$-algebras can be constructed via (co)localizations of $\iNSp^\op$.
\end{rem}

\begin{lem} \label{Brown}
The stable $\infty$-categories $\Sp(\iNS)$ and $\iNSp$ are compactly generated. 
\end{lem}

\begin{proof}
We have already seen that $\iNS$ is compactly generated (see Remark \ref{compGen}). Hence by Proposition 1.4.3.7 of \cite{LurHigAlg}, so is $\Sp(\iNS)$. Now the assertion for $\iNSp$, which is an accessible localization of $\Sp(\iNS)$, follows from Corollary 5.5.7.3 of \cite{LurToposBook}. \footnote{Note that the domains and the codomains of the morphisms in $S_0$ are all $\omega$-compact.}
\end{proof}

It also follows from Proposition 1.4.3.7 of \cite{LurHigAlg} that the essential image of the functor $\Pi:\hosc[\Sigma^{-1}]\functor\hSp(\iNS)^\op$  (see Proposition \ref{Ess}) lands inside the cocompact objects of $\hSp(\iNS)^\op$. The deployment of heavy machinery pays rich dividends at this point, viz., using Lemma \ref{Brown} we find that the triangulated category $\hNSp$ satisfies Brown representability. This is a major advantage of $\hNSp$ over $\NSH$ or other suspension stable homotopy categories of $C^*$-algebras like $\SHo$ of \cite{CunMeyRos}. Sometimes in the literature one studies Brown representability for abelian group valued functors that behave in a specific manner with respect to weak colimits. Here we use the formulation from Theorem 1.4.1.2 of \cite{LurHigAlg}.

\begin{thm}[Brown representability] \label{BR}
A functor $F:\hNSp^\op\functor\mathtt{Set}$ is representable if and only if it satisfies the following two conditions:

\begin{itemize}
 \item The canonical map $F(\coprod_\beta C_\beta)\map \prod_\beta F(C_\beta)$ is a bijection for every collection of objects $C_\beta\in\iNSp$, and
 \item for every pushout square 
 
 \beqn
 \xymatrix{
 C \ar[r]\ar[d] & C'\ar[d] \\
 D \ar[r] & D',
 }\eeqn in $\iNSp$, the induced map $F(D')\map F(C')\times_{F(C)} F(D)$ is surjective.
\end{itemize}
\end{thm}

\begin{cor} \label{ProCopro}
 The triangulated category $\hNSp$ has arbitrary coproducts and products.
\end{cor}

\begin{proof}
Existence of coproducts follows from the fact that $\iNSp$ is presentable. To see the existence of products using Brown representability argue as follows: For any collection of objects $X_\alpha\in\hNSp$ the functor $\prod_\alpha \hNSp(-,X_\alpha)$ is representable and the representing object is the desired product $\prod_\alpha X_\alpha$. It can also be deduced from the presentability of $\iNSp$, since presentable $\infty$-categories admit all small limits.
\end{proof}

\begin{rem} \label{dualBrown}
If a triangulated category $\cT$ is compactly generated, then it does not imply that $\cT^\op$ is also compactly generated. In fact, if a compactly generated triangulated category $\cT$ satisfies some extra hypotheses, then one can show that $\cT^\op$ is not well generated (see Appendix E of \cite{NeeBook}). Nevertheless, the opposite of a compactly generated triangulated category satisfies Brown representability \cite{KraBR}; hence $\hNSp^\op$ satisfies Brown representability.
\end{rem}

\begin{thm} \label{ExInc}
The noncommutative stable homotopy category $\NSH$ can be realised as a full triangulated subcategory of $\hNSp^\op$ via a canonical exact functor.
\end{thm}

\begin{proof}
We know from Proposition \ref{Ess} that the functor $\Pi:\hosc[\Sigma^{-1}]\functor\h\Sp(\iNS)^\op$ is fully faithful. Due to the compact generation of the triangulated category $\h\Sp(\iNS)$ (see Lemma \ref{Brown}), its localization with respect to the set of maps $S$ (described in the second paragraph of subsection \ref{NSpLoc}) can be viewed as a Bousfield localization $L:\h\Sp(\iNS)\functor\h\Sp(\iNS)$, whose essential image consists of $S$-local objects \cite{NeeBook}. This is precisely the triangulated category $\hNSp$. Consider the (solid) commutative diagram of triangulated categories:

\beq \label{VerLoc}
\xymatrix{
\ker(V)\ar[rr]^\Pi\ar[d] && \ker(L^\op) \ar[d] \\
\hosc[\Sigma^{-1}]\ar[d]_V\ar[rr]^{\Pi} &&\hSp(\iNS)^\op\ar[d]^{L^\op}\\
\NSH\ar@{-->}[rr]^{\pi} && \hNSp^\op,
}\eeq where $V$ denotes the Verdier quotient functor. The composite functor $L^\op\circ\Pi$ inverts the maps in $\{\theta(f):\ker(f)\map\C(f)\,|\, \text{$f$ surjective $*$-homomorphism in $\Csep$}\}^\op$ since $L^\op$ inverts $S_0^\op$. Thus it annihilates $\ker(V)$, which is the thick subcategory of $\hosc[\Sigma^{-1}]$ generated by the set $X=\{\textup{cone}(\theta(f)) \,|\, \text{ $f$ surjective in $\Csep$}\}$. Consequently, $L^\op\circ\Pi$ induces a unique (dashed) functor $\pi:\NSH\functor\hNSp^\op$, making the above diagram commute. In order to proceed it is useful to consider the opposite of the above diagram \eqref{VerLoc}:

\beq
\xymatrix{
\ker(V^\op)\ar[rr]^{\Pi^\op}\ar[d] && \ker(L) \ar[d] \\
\hosc[\Sigma^{-1}]^\op\ar[d]_{V^\op}\ar[rr]^{\Pi^\op} &&\hSp(\iNS)\ar[d]^{L}\\
\NSH^\op\ar@{-->}[rr]^{\pi^\op} && \hNSp.
}\eeq For any triangulated category $\cT$ let $\cT^c\subset\cT$ denote the thick subcategory of compact objects. Now $\ker(L)$ is the localizing subcategory of $\hSp(\iNS)$ generated by the set $\Pi^\op(X)$. The essential image $\Pi^\op(\hosc[\Sigma^{-1}]^\op)$ is contained inside $\h\Sp(\iNS)^c$ (see Proposition \ref{Ess}). Using Lemma \ref{Brown} and Corollary 7.2.2 of \cite{KraLoc} one deduces that $\ker(L)^c = \ker(L)\cap \hSp(\iNS)^c$; moreover, $\ker(L)^c$ is the thick closure of $\Pi^\op(\ker(V^\op))$ inside $\ker(L)$. Thus we obtain the following commutative diagram:

\beqn
\xymatrix{
\ker(V^\op)\ar[rr]^{\Pi^\op}\ar[d] && \ker(L)^c\ar[rr]^{\subset}\ar[d]^{} && \ker(L) \ar[d] \\
\hosc[\Sigma^{-1}]^\op\ar[d]_{V^\op}\ar[rr]^{\Pi^\op} && \hSp(\iNS)^c \ar[rr]^\subset\ar[d] &&\hSp(\iNS)\ar[d]^{L}\\
\NSH^\op\ar[rr]^{\theta_1} && \hSp(\iNS)^c/\ker(L)^c\ar[rr]^{\theta_2} && \hNSp,
}\eeqn such that $\theta_2\circ\theta_1 =\pi^\op$. From Theorem 7.2.1 (3) of \cite{KraLoc} one concludes that $\theta_2$ is fully faithful and from Lemma 4.7.1 (1) of \cite{KraLoc} one concludes that $\theta_1$ is fully faithful. Thus $\pi^\op$ is fully faithful and hence so is $\pi:\NSH\functor\hNSp^\op$.  
\end{proof}

Recently Schwede defined a {\em topological} triangulated category to be one, which is triangle equivalent to the homotopy category of a {\em stable cofibration category} as a triangulated category (see Definition 1.4 of \cite{SchTopTri}). The Theorem \ref{ExInc} above already says that $\NSH$ is morally a topological triangulated category. Now we make it precise in terms of the above definition.

\begin{thm} \label{TopTri}
The triangulated category $\NSH$ is topological.
\end{thm}

\begin{proof}
Every full triangulated subcategory of a topological triangulated category is itself topological (see Proposition 1.5 of \cite{SchTopTri}). Since we just showed that $\NSH$ is a full triangulated subcategory of $\hNSp^\op$, it suffices to show that $\hNSp^\op$ is topological. Now $\hNSp^\op$ is itself a full triangulated subcategory of $\hSp(\iNS)^\op$. Since $\Sp(\iNS)$ is a presentable $\infty$-category, there is a combinatorial simplicial model category $\cA$, whose underlying $\infty$-category is equivalent to $\Sp(\iNS)$ (see Proposition A.3.7.6 of \cite{LurToposBook}). Endow $\cA^\op$ with the opposite model structure and consider the underlying cofibration category, i.e., the full subcategory $\cA^\op_c$ consisting of all the cofibrant objects with the associated weak equivalences and cofibrations. The cofibration category $\cA^\op_c$ is stable owing to the stability of $\Sp(\iNS)^\op$ and its homotopy category is equivalent to $\hSp(\iNS)^\op$ by construction. This argument can be used to also prove that $\NSH^\op$ is topological.
\end{proof}

\begin{rem} \label{locallyConvex}
The $\infty$-category $\iCsep$ (resp. $\iNS$) is an ideal framework for homotopy theory of separable $C^*$-algebras (resp. pointed noncommutative spaces). In fact, the methodology is applicable to a more general class of algebras at the expense of added complexity. Let us indicate how to proceed in the case of locally convex algebras. Let $\LC$ be the category of locally convex algebras and $\LW$ be the subcategory of {\em diffotopy equivalences} (see Definition 4.1.3 of \cite{CunTho}). Applying the Dwyer--Kan simplicial localization to the pair $(\LC,\LW)$ produces a simplicial category \cite{DwyKan}, whose homotopy coherent nerve is a simplicial set. We can take the fibrant replacement of this simplicial set with respect to the Joyal model structure on simplicial sets (see Theorem 2.4.6.1 of \cite{LurToposBook}) to obtain an $\infty$-category $\LC_\infty$, that serves as a good model. There are some other natural candidates like algebraic homotopy equivalences that may replace diffotopy equivalences in the above construction.
\end{rem}

\begin{rem} \label{EbuKK}
We constructed stabilizations of the category of separable $C^*$-algebras and noncommutative pointed spaces as stable $\infty$-categories, namely, $\Sp(\iCsep)$ and $\iNSp$ respectively. Actually we can construct more interesting stable $\infty$-categories via (co)localizations. We let $M_1=\{A\map A\prot\cpt\,|\, A\in\Csep\}$, $M_2=\{A\map A\prot M_2(A)\,|\, A\in\Csep\}$ and $M_1^\op$, $M_2^\op$ denote the corresponding sets of morphisms in the opposite category $\Csep^\op$. For any set $T$ of morphisms in $\Sp(\iNS)$ let $\langle T\rangle$ denote the strongly saturated collection generated by $T$ by the procedure described in the second paragraph of subsection \ref{NSpLoc}. Localizing $\Sp(\iNS)$ with respect to the following strongly saturated collections of morphisms
 \begin{itemize}
  \item $\langle S\cup \St(M_1^\op)\rangle$,
  \item $\langle S\cup\St(M_2^\op)\rangle$,
 \end{itemize} one obtains stable $\infty$-categories, whose opposite stable $\infty$-categories are our models for 
 \begin{itemize}
  \item $\E$-theory (denoted by $\mathtt{E_\infty}$),
  \item connective $\E$-theory or $\bu$-theory (denoted by $\mathtt{bu_\infty}$)
  \end{itemize} respectively. If one localizes $\Sp(\iNS)$ with respect to $\langle S'\cup\St(M_1^\op)\rangle$, where $S'\subset S$ arises from those surjections in $\Csep$ that admit a completely positive contractive splitting, then one obtains yet  another stable $\infty$-category. Its opposite stable $\infty$-category, denoted by $\mathtt{KK}_\infty$, is our model for $\KK$-theory. Since $S'\subset S$, one obtains a canonical exact $\infty$-functor $\mathtt{KK}_\infty\functor\mathtt{E}_\infty$. One could also localize $\Sp(\iNS)$ with respect to $\langle S'\rangle$ only. The opposite of this localized stable $\infty$-category will model the triangulated category $\SHo$ that appeared in \cite{CunMeyRos}; more precisely, $\SHo$ will be contained as a full triangulated subcategory of its homotopy category. All triangulated categories in sight will be topological.
\end{rem}

\subsection{A comparison between $\NSH$ and {\O}stv{\ae}r's stable homotopy category} \label{Ostvaer}
Let us recall from \cite{ThomThesis} that a covariant functor from $\Csep$ (not viewed as a topological category) to a triangulated category $\cT$ with suspension functor $\Sigma$ is called a {\em triangulated homology theory} if it is homotopy invariant and for every short exact sequence 
\beq
0\map A\map B\map C\map 0
\eeq in $\Csep$ the diagram induced by $H$ 

\beq \label{exSeq}
H(A)\map H(B)\map H(C)\map\Sigma H(A)
\eeq is an exact triangle in $\cT$. Furthermore, the exact triangle \eqref{exSeq} should be natural with respect to morphisms of exact sequences. The canonical functor $\iota:\Csep\functor\NSH$ can be characterized as the universal triangulated homology theory (see Theorem 3.3.6 of \cite{ThomThesis}). 

In \cite{Ost} the author constructs the stable model category of $C^*$-algebras as follows: first an unstable model category containing $C^*$-algebras is constructed as a cubical set valued presheaf category. Then one considers (bigraded) spectrum objects over the unstable model category to produce a stable model category. Finally the stable homotopy category $\SH^*$ is obtained as a localization, which introduces the right formal properties that one should expect from this motivic setup. It follows from the above characterization of $\NSH$ as the universal triangulated homology theory that there is a canonical exact functor $C:\NSH\functor\SH^*$ (see Corollary 4.43 of \cite{Ost}). The functor $C$ is reminiscent of the functor $c:\mathcal{SH}\functor\mathcal{SH}(\CC)$ from the classical stable homotopy category to the motivic stable homotopy category (see, for instance, \cite{Levine}) and as such should be studied more thoroughly.

\section{Noncommutative stable (co)homotopy groups and (co)limits} \label{Htpy}
Using the stable $\infty$-category of noncommutative spectra $\iNSp$ one can produce numerous generalized (co)homology theories for $C^*$-algebras. We are going to investigate them in our subsequent work. Here we discuss the most fundamental example, i.e., noncommutative stable (co)homotopy. It is (co)represented by the (noncommutative) sphere spectrum $\Sigma^\infty_S (\CC)$.

Consider the functors $\NSH(-,\Sigma^n\CC)$ for all $n\in\NN$ on the category $\Csep$. If we insert $\C(X,x)$ into the first variable, where $(X,x)$ is a finite pointed CW complex, the functors give the stable homotopy groups $\pi_n(X,x)$ naturally. Similarly the functors $\NSH(\CC, \Sigma^n(-))$ generalize the stable cohomotopy groups of finite pointed CW complexes. These functors give reasonable definitions of noncommutative stable (co)homotopy groups for separable $C^*$-algebras \cite{ConBook}. A natural question is how to proceed beyond separable $C^*$-algebras. Our methodology provides a satisfactory answer to this question.

\begin{rem} \label{CunMeyRos}
In \cite{CunMeyRos} the authors constructed a suspension stable homotopy category $\SHo$ for all (possibly nonseparable) $C^*$-algebras. Unfortunately, in this triangulated category one has the following problem (see Section 6.3.1 of \cite{CunMeyRos}): $$\SHo(\oplus_{n\in\NN} A_n, B)\ncong \prod_{n\in\NN} \SHo(A_n,B).$$ Hence the authors in \cite{CunMeyRos} remarked that it may not be the right definition for the suspension stable homotopy category beyond separable $C^*$-algebras.
\end{rem}

\noindent
In contrast our triangulated homotopy category of noncommutative spectra $\hNSp$ admits arbitrary coproducts and products (see Corollary \ref{ProCopro}). Thus it is more natural to introduce the noncommutative stable (co)homotopy groups of noncommutative pointed spaces using noncommutative spectra. In the sequel we show that in this setup it is quite simple to describe the behaviour of the resulting theory under sequential (co)limits.

\subsection{Noncommutative stable (co)homotopy groups} 
In view of Theorem \ref{ExInc} and Remark \ref{Pi} we propose the following definition of noncommutative stable (co)homotopy groups of noncommutative spaces. We leave out the corresponding definitions for noncommutative spectra, which are absolutely clear.

\begin{defn}
 For any pointed noncommutative space $A\in\iNS$ we define
 \beqn
\hNSp(\Sigma^\infty_S (A), \Sigma^\infty_S(\CC)) &=& \text{noncommutative stable cohomotopy group of $A$, and}\\
\hNSp(\Sigma^\infty_S(\CC),\Sigma^\infty_S(A)) &=& \text{noncommutative stable homotopy group of $A$.}
\eeqn The higher and the lower noncommutative stable (co)homotopy groups are defined in a predictable manner by suspending the variables appropriately.

\noindent
For any separable $C^*$-algebra $A$, we set for all $n\geqslant 0$
\beqn
\pi^n(A) &:=&\hNSp(\Sigma^\infty_S(\Sigma^n\CC),\Sigma^\infty_S(A))\cong\NSH(A,\Sigma^n\CC), \text{ and}\\
\pi_n(A) &:=&\hNSp(\Sigma^\infty_S(\Sigma^n A),\Sigma^\infty_S(\CC))\cong\NSH(\CC,\Sigma^n A).
\eeqn If $n < 0$ one puts the suspension functor in the other coordinate to define the corresponding groups. One recovers via $\pi_n(A)$ (resp. $\pi^n(A)$) the $n$-th noncommutative stable cohomotopy (resp. noncommutative stable homotopy) group of $A$ as defined in \cite{ThomThesis}.
\end{defn}

\begin{rem} \label{nonsep}
Our formalism is able to treat the noncommutative stable (co)homotopy of nonseparable $C^*$-algebras (or noncommutative pointed compact Hausdorff spaces) as well. Indeed, any nonseparable $C^*$-algebra can be written as a filtered colimit of its separable $C^*$-subalgebras. One can view this as a projective system in $\iNS$ and apply the above definitions to its $\infty$-limit in $\iNS$. Our noncommutative stable (co)homotopy groups for genuinely nonseparable $C^*$-algebras will in general not agree with those defined in \cite{CunMeyRos}. 
\end{rem}

\begin{rem} \label{M2unstable}
Noncommutative stable homotopy theory does not satisfy matrix stability or $C^*$-stability. It is known that $\pi^0(M_2(\CC))\simeq 0$, whereas $\pi^0(\CC)\simeq \ZZ$ \cite{MyNGH}. Nevertheless, $M_2(\CC)$ is not a zero object in $\NSH$. If it were a zero object in $\NSH$, then after an exact localization it would continue to be a zero object in the bivariant $\E$-theory category. However, it is easily seen to be a non-zero object in the bivariant $\E$-theory category because $\E_0(M_2(\CC))\simeq\ZZ$.
\end{rem}

\subsection{Behaviour with respect to (co)limits} 

\noindent
There are stable (co)homotopy groups for general pointed compact Hausdorff spaces that are studied in shape theory (see, for instance, \cite{MarSeg}). One of the standard techniques to compute them is to break up a complicated (noncommutative) space into a diagram of small computable (noncommutative) spaces. For the benefit of the reader we illustrate this principle by recalling a basic property of stable cohomotopy groups from shape theory in the dual setting of commutative $C^*$-algebras.

\begin{thm}[see, for instance, \cite{Nowak}] \label{Cont}
Let $\{A_n\}_{n\in\NN}$ be a countable inductive system of commutative separable $C^*$-algebras. Then for all $k\in\ZZ$ one has an isomorphism $$\pi_k({\dlim}_n A_n)\cong{\dlim}_n\pi_k(A_n).$$
\end{thm}

\begin{proof}
Let us first assume that $k\geqslant 0$. Using Corollary 17 of \cite{DadAsymHom} we may conclude that $\pi_k({\dlim}_n A_n)\cong{\dlim}_m [\Sigma^m\CC,\Sigma^{m+k}{\dlim}_n A_n]$. Since $\Sigma^m\CC$ is $\prot$-continuous for all $m\in\NN$ we have $[\Sigma^m\CC,\Sigma^{m+k}{\dlim}_n A_n]\cong [\Sigma^m\CC,{\dlim}_n\Sigma^{m+k} A_n]$. Moreover, $\Sigma^m\CC$ is semiprojective in the category of commutative separable $C^*$-algebras for all $m\in\NN_{\geqslant 1}$, whence $[\Sigma^m\CC,{\dlim}_n\Sigma^{m+k} A_n]\cong{\dlim}_n [\Sigma^m\CC,\Sigma^{m+k} A_n]$. Thus we conclude

\beqn
\pi_k({\dlim}_n A_n) &=& {\dlim}_m[\Sigma^m\CC,\Sigma^{m+k}{\dlim}_n A_n],\\ 
&\cong&{\dlim}_m [\Sigma^m\CC,{\dlim}_n \Sigma^{m+k}A_n],\\
 &\cong& {\dlim}_m{\dlim}_n [\Sigma^m\CC,\Sigma^{m+k} A_n],\\
 &\cong& {\dlim}_n{\dlim}_m [\Sigma^m\CC,\Sigma^{m+k} A_n],\\
 &=& {\dlim}_n \pi_k (A_n).
\eeqn If $k<0$ then one puts the suspension $\Sigma^k$ in the first variable and argues similarly.
\end{proof}

\begin{rem} \label{Dad}
The above argument does not generalize to all noncommutative separable $C^*$-algebras, since $\Sigma^m\CC$ fails to be semiprojective in the category of all separable $C^*$-algebras for $m>1$ \cite{SorThi}.
\end{rem}

\begin{cor}
The noncommutative stable cohomotopy groups of a commutative separable $C^*$-algebra are countable in each degree. 
\end{cor}

\begin{proof}
Let us first assume that $A$ is a separable commutative unital $C^*$-algebra. Its spectrum $X$ is a compact metrizable space. Being a compact metrizable space, it admits a description as a countable inverse limit of finite complexes, i.e., $X\cong\ilim_n X_n$ with each $X_n$ a finite complex (see, for instance, Theorem 7 of \cite{MarSeg}). This inverse limit expresses $\C(X)$ as a direct limit of commutative $C^*$-algebras ${\dlim}_n \C(X_n)$, whose noncommutative stable cohomotopy groups can be computed using Theorem \ref{Cont}, i.e., $\pi_*(\C(X))\cong{\dlim}_n \pi_*(\C(X_n))$. Note that $\pi_*(\C(X_n))$ is computable in terms of stable cohomotopy of the finite complex $X_n$. Moreover, the stable cohomotopy groups of a finite complex are countable in each degree. The general case involving nonunital $C^*$-algebras now follows by excision.
\end{proof}

\noindent
Since $\iNSp$ is presentable, it admits all small limits and colimits. For simplicity we only treat the case of countable filtered (co)limits or sequential (co)limits in the sequel. In this case there is an elementary notion of a homotopy (co)limit in a triangulated category with (co)products \cite{BokNee}, which suffices for our purposes. Let $\cT$ be a triangulated category with products and let $\Sigma$ denote its suspension functor. Let $\{A_n,\alpha_n\}_{n\in\NN}$ be a countable inverse system in $\cT$. Then its homotopy limit $\holim A_n$ is defined by the following exact triangle: \beq \label{holim} \holim A_n\map {\prod}_{n=0}^\infty A_n\overset{1-\mu}{\map} {\prod}_{n=0}^\infty A_n\map \Sigma\holim A_n .\eeq Here the map $\mu: \prod_{n=0}^\infty A_n\map \prod_{n=0}^\infty A_n$ sends the factor $A_n$ in the domain to the factor $A_{n-1}$ in the codomain via $\alpha_n$. Dually, in a triangulated category with coproducts the homotopy colimit $\hocolim A_n$ of a directed system $\{A_n,\alpha_n\}_{n\in\NN}$ is defined by the following exact triangle: \beq \label{hocolim} \Sigma^{-1}\hocolim A_n\map {\coprod}_{n=0}^\infty A_n\overset{1-\mu}{\map} {\coprod}_{n=0}^\infty A_n\map \hocolim A_n .\eeq Here the map $\mu: \coprod_{n=0}^\infty A_n\map \coprod_{n=0}^\infty A_n$ is defined similarly.

\begin{prop}
 Let $\{A_n,\alpha_n\}_{n\in\NN}$ be a countable inverse system in $\hNSp$. Then for all $B\in\hNSp$ there is a short exact sequence of abelian groups $$ 0\map {\ilim}^1\hNSp(B,\Sigma^{-1} A_n) \map\hNSp(B,\holim A_n)\map\ilim\hNSp(B, A_n)\map 0.$$
\end{prop}

\begin{proof}
We include a proof to emphasize the simplicity of the arguments in this setup. Applying the functor $\hNSp(B,-)$ to the exact triangle \eqref{holim} we get a long exact sequence: $$\hNSp(B,\holim A_n)\map {\prod}_{n=0}^\infty\hNSp(B, A_n)\overset{1-\mu}{\map}{\prod}_{n=0}^\infty \hNSp(B, A_n)\map \hNSp(B,\Sigma\holim A_n) .$$ Thus we get a short exact sequence $$0\map\C \map \hNSp(B,\holim A_n) \map \K\map 0, \text{ where} $$ $$\K = \ker (\prod_{n=0}^\infty \hNSp(B,A_n)\overset{1-\mu}{\map}\prod_{n=0}^\infty \hNSp(B,A_n)) \text{ and}$$ $$\C = \coker (\prod_{n=0}^\infty \hNSp(B,\Sigma^{-1}A_n)\overset{1-\mu}{\map}\prod_{n=0}^\infty \hNSp(B,\Sigma^{-1}A_n)).$$ It is easy to see that $\K=\ilim\hNSp(B,A_n)$ and $\C=\ilim^1 \hNSp(B,\Sigma^{-1} A_n)$.
\end{proof}

\begin{prop} \label{lim1}
 Let $\{A_n,\alpha_n\}_{n\in\NN}$ be a countable directed system in $\hNSp$. Then for all $B\in\hNSp$ there is a short exact sequence of abelian groups $$ 0\map {\ilim}^1\hNSp(\Sigma A_n,B) \map\hNSp(\hocolim A_n,B)\map\ilim\hNSp(A_n,B)\map 0.$$
\end{prop}

\begin{proof}
 Applying the functor $\hNSp(-,B)$ to the exact triangle \eqref{hocolim} we get a long exact sequence: $$\hNSp(\hocolim A_n,B)\map{\prod}_{n=0}^\infty \hNSp( A_n,B)\overset{1-\mu}{\map} {\prod}_{n=0}^\infty \hNSp(A_n,B)\map \hNSp(\Sigma^{-1}\hocolim A_n,B) .$$ Thus we get a short exact sequence $$0\map\C \map \hNSp(B,\holim A_n) \map \K\map 0, \text{ where}$$ $$\K = \ker (\prod_{n=0}^\infty \hNSp(A_n,B)\overset{1-\mu}{\map}\prod_{n=0}^\infty \hNSp(A_n,B)) \text{ and}$$ $$\C = \coker (\prod_{n=0}^\infty \hNSp(\Sigma A_n,B)\overset{1-\mu}{\map}\prod_{n=0}^\infty \hNSp(\Sigma A_n,B)).$$ Again it is easy to see that $\K=\ilim\hNSp(A_n,B)$ and $\C=\ilim^1 \hNSp(\Sigma A_n,B)$.
\end{proof}

\begin{prop} \label{lim2}
 Let $\{A_n,\alpha_n\}_{n\in\NN}$ be a countable directed system in $\hNSp$. Then for every compact object $B\in\hNSp$ one has an isomorphism $$\dlim\hNSp(B, A_n)\overset{\sim}{\map}\hNSp(B,\hocolim A_n).$$
 
 \noindent
 Dually, if $\{A_n,\alpha_n\}_{n\in\NN}$ is a countable inverse system in $\hNSp$, then for every cocompact object $B\in\hNSp$ one has an isomorphism $$\dlim\hNSp(A_n,B)\overset{\sim}{\map}\hNSp(\holim A_n,B).$$
\end{prop}

\begin{proof}
The assertions follow from Lemma 1.5 of \cite{NeeCompObj} (and its dual).
\end{proof}

\begin{ex}
 Let $A$ be any separable $C^*$-algebra. Then $\st(A)=\Sigma^\infty_S(j(A))$ is a compact object in $\iNSp$, which descends to a compact object in $\hNSp$.
\end{ex}

\begin{rem}
In fact, using $\infty$-colimits one can say something sharper. Let $\cI$ be a small filtered $\infty$-category and let $p:\cI\map\iNSp$ be an $\infty$-functor. We denote by $\dlim (p)$ the filtered $\infty$-colimit of this diagram in $\iNSp$. If $B$ is a compact object in $\iNSp$ then, by definition, there is an equivalence of spectra $\dlim_{\cI} \iNSp(B,p(i))\simeq\iNSp(B,\dlim(p))$. A dual of this statement refining the isomorphism $\dlim\hNSp(A_n,B)\overset{\sim}{\map}\hNSp(\holim A_n,B)$ above is also valid.
\end{rem}

Let $\Csep^\delta$ temporarily denote the category of separable $C^*$-algebras (with discrete morphism spaces) and let $\Csep$ denote the topological category of separable $C^*$-algebras. We know that being presentable $\iNS$ admits all small $\infty$-(co)limits whence so does $\iNS^\op$. In a subsequent project we are going to construct a model structure on $\Pro_\omega(\Csep^\delta)$, whose homotopy category will contain $\hosc$ as a full subcategory. Moreover, we are going to relate the homotopy (co)limits in it with the $\infty$-(co)limits in $\Pro_\omega(\iCsep)\simeq\iNS^\op$.

Recall from Definition \ref{NSpectra} that there is a stabilization $\infty$-functor $\Sigma^\infty_S :\iNS\functor\iNSp$. Being a composition $L\circ\Sigma^\infty$ of two left adjoint functors, the $\infty$-functor $\Sigma^\infty_S$ preserves colimits.

\begin{thm} \label{hNSpCont}
  Let $\{C_n,\beta_n: C_{n}\map C_{n+1}\}_{n\in\NN}$ be a countable directed diagram in $\iNS$. Then for every $B\in\hNSp$ one has $$ 0\map {\ilim}^1\hNSp(\Sigma(\Sigma^\infty_S (C_n)),B) \map\hNSp(\Sigma^\infty_S (\dlim C_n),B)\map\ilim\hNSp(\Sigma^\infty_S (C_n),B)\map 0.$$ Dually, for every compact object $B\in\hNSp$ one has $$\dlim\hNSp(B, \Sigma^\infty_S (C_n))\overset{\sim}{\map}\hNSp(B,\Sigma^\infty_S (\dlim C_n)).$$
\end{thm}

\begin{proof}
Since $\Sigma^\infty_S$ preserves colimits we have $\Sigma^\infty_S (\dlim C_n)\simeq\dlim \Sigma^\infty_S (C_n)$, which descends to a homotopy colimit in $\hNSp$. The assertions now follow from Propositions \ref{lim1} and \ref{lim2}.
\end{proof}

\begin{lem}
For any nuclear separable $C^*$-algebra $A$, there is a natural homomorphism $\pi_*(A)\map\K_*(A)$ induced by the corner embedding $A\map A\prot\cpt$. This natural homomorphism is an isomorphism if $A$ is stable.
\end{lem}

\begin{proof}
Consider the natural homomorphism $\pi_*(A)\map\pi_*(A\prot\cpt)$ induced by the corner embedding. Since the noncommutative stable cohomotopy groups coincide with the $\E$-theory groups for stable $C^*$-algebras, one may identify $\pi_*(A\prot\cpt)\cong\E_*(A\prot\cpt)$. Moreover, the nuclearity of $A$ implies that one can naturally identify $\E_*(A\prot\cpt)\cong\K_*(A\prot\cpt)\cong\K_*(A)$. Here the last identification is due to the $C^*$-stability of topological $\K$-theory, which is natural. The second assertion is clear.
\end{proof}

\begin{ex}
 In Example 11 of \cite{MyTwist} it is explained how one can construct an inverse system of separable $C^*$-algebras $\{C_n\}_{n\in\NN}=\{(\textup{CT}(SU(n)),\iota_n^*(P))\}_{n\in\NN}$ starting from a principal $PU$-bundle $P$ on $SU(\infty)$. The inverse limit of this diagram in topological $*$-algebras is the {\em noncommutative twisted version of $SU(\infty)$}. Using the above Lemma and Theorem \ref{hNSpCont} we deduce that the noncommutative stable cohomotopy groups of $\dlim C_n$ in $\iNS$ vanish.
\end{ex}

\subsection{Finite group $C^*$-algebras}
For a finite group (or, more generally, a compact group) $G$ and a $G$-$C^*$-algebra $A$, the {\em Green--Julg--Rosenberg Theorem} establishes a natural isomorphism $\K^G_*(A)\cong\K_*(A\rtimes G)$. In particular, setting $A=\CC$ we find that the $\K$-theory of the group $C^*$-algebra $C^*(G)$ is isomorphic to the $G$-equivariant $\K$-theory of a point. One might wonder whether the pattern persists in noncommutative stable cohomotopy. This would show that noncommutative stable cohomotopy of a finite group $C^*$-algebra is isomorphic to equivariant stable cohomotopy of the $0$-sphere, which in turn can be computed using the Segal conjecture. Unfortunately, the answer turns out be negative as we presently demonstrate.

\begin{ex} \label{GJR}
It is known that the $G$-equivariant $0$-th stable (co)homotopy group of a point is isomorphic to the Burnside ring of $G$. The underlying abelian group of the Burnside ring is generated by the finite set $\{G/H\, |\, H\subseteq G \text{ subgroup}\}$. Noncommutative stable cohomotopy is finitely additive, i.e., one has $\pi_*(\prod_{i=1}^n A_n)\cong\oplus_{i=1}^n\pi_*(A_n)$. 

Let $G=\ZZ/p$, where $p$ is an odd prime. Then $C^*(G)$ is a commutative finite dimensional $C^*$-algebra, whence it decomposes as $C^*(G)\cong \prod_{i=1}^p \CC$. Using the finite additivity of noncommutative stable cohomotopy and the fact that $\pi_0(\CC)\cong\pi^0(S^0)\simeq\ZZ$, we see that $\pi_0(C^*(G))\simeq \ZZ^p$. Now one immediately observes that the rank of $\pi_0(C^*(G))$ differs from the rank of the Burnside ring of $G$, which is $2$ since $G=\ZZ/p$ is a simple group.
\end{ex}

\noindent
In general, one can reduce the problem to the computation of the noncommutative stable cohomotopy groups of matrix algebras using Maschke and Artin--Wedderburn Theorems. 

\begin{lem} \label{FinComp}
Let $G$ be a finite group, so that $C^*(G)\cong\prod_{i=1}^k M_{n_i}(\CC)$. Then one has an isomorphism $\pi_m(C^*(G))\cong\oplus_{i=1}^k \pi_m(M_{n_i}(\CC))$ as abelian groups.
\end{lem}

\begin{proof}
The assertion follows immediately from the finite additivity of $\pi_m(-)$.
\end{proof}

\noindent
Observe that if $\Csep(\C(X,x),M_n(\CC))$ is an ANR with the point-norm topology, then the canonical map $$c:[\C(X,x),\C(Y,y)\otimes M_n(\CC)]\map[[\C(X,x),\C(Y,y)\otimes M_n(\CC)]]$$ is an isomorphism (see Proposition 16 of \cite{DadAsymHom}).

\begin{lem} \label{ANR}
The space $\Csep(\Sigma^r\CC,M_n(\CC))$ is an ANR for all $r,n\in\NN$.
\end{lem}

\begin{proof}
There is a homeomorphism of spaces $\Csep(\Sigma^r\CC,M_n(\CC))\cong\mathtt{SC^*_1}(\C(S^r),{M}_n(\CC))$ (see 1.2.4. of \cite{DadNem}), where $\mathtt{SC^*_1}(\C(S^r),{M}_n(\CC))$ denotes the space of unital $*$-homomorphisms between unital $C^*$-algebras. Now $\C(S^r)$ is the universal $C^*$-algebra on a finite set of generators and relations, i.e., $\C(S^r)\cong C^*\{x_1,\cdots,x_r\;|\: x_i =x^*_i, x_ix_j=x_jx_i, \sum_{i=1}^r x_i^2 =1\}$ (see, for instance, Theorem 1.1. of \cite{BanGos}). It follows that $\mathtt{SC^*_1}(\C(S^r),M_n(\CC))$ is a compact finite dimensional manifold (in fact, a real algebraic variety with standard topology). It is well-known that such a space is an ANR (see, for instance, Theorem 26.17.4 of \cite{GreHar}).
\end{proof}

\begin{prop}
 The noncommutative stable cohomotopy groups of $M_n(\CC)$ are
 \beqn
\pi_k(M_n(\CC))\cong\begin{cases}{\dlim}_r [\Sigma^r\CC,\Sigma^{r+k} M_n(\CC)] & \text{if $k\geqslant 0$,}\\
                {\dlim}_r [\Sigma^{r+k}\CC,\Sigma^r M_n(\CC)] & \text{if $k<0$.}
               \end{cases}
\eeqn
\end{prop}

\begin{proof}
Since the following diagram commutes for all $r\in\NN$
\beqn
\xymatrix{
[\Sigma^r\CC,\Sigma^{r+k} M_n(\CC)]\ar[r]\ar[d]^c & [\Sigma^{r+1}\CC,\Sigma^{r+1+k} M_n(\CC)]\ar[d]^c\\
[[\Sigma^r\CC,\Sigma^{r+k} M_n(\CC)]]\ar[r] & [[\Sigma^{r+1}\CC,\Sigma^{r+1+k} M_n(\CC)]]
}
\eeqn and the vertical arrows are isomorphisms due to the above Lemma \ref{ANR}, the assertion follows from Equation \eqref{PiDef} by passing to the inductive limit.
\end{proof}

\begin{rem}
There is a spectrum $\{X_n\}$ with $X_n = \mathtt{SC^*_1}(\C(S^r),M_n(\CC))$, whose homotopy groups are the noncommutative stable cohomotopy groups of $M_n(\CC)$. The proof of Lemma \ref{ANR} gives an explicit description of the spaces $X_n$ as real algebraic varieties and it is plausible that the homotopy groups of these spaces can be computed using algebro-geometric methods.
\end{rem}

\begin{ex}
Let us consider the case of $2\times 2$-matrices. There is a canonical embedding $\iota:\CC^2\map M_2(\CC)$, where $\CC^2$ is viewed as the diagonal $C^*$-algebra. The mapping cone $\C(\iota)$ can be identified with $q\CC$ (see, for instance, Section 3.3 of \cite{Loring}). Therefore, the noncommutative stable cohomotopy groups of $M_2(\CC)$ can be computed from those of $\CC^2$ and $q\CC$ via the induced long exact sequence. There is also a short exact sequence of $C^*$-algebras $0\map q\CC\map \CC\ast\CC\map\CC\map 0$ by the very defintion of $q\CC$ \cite{CunKK}. It follows that the noncommutative stable cohomotopy groups of $q\CC$ can be read off from the induced long exact sequence in terms of those of $\CC\ast\CC$ and $\CC$. The computation is further facilitated by the fact that the extension $0\map q\CC\map\CC\ast\CC\map\CC\map 0$ is actually split (via two canonical splittings). Note that $\CC\cong \C(S^0)$, where $S^0$ is the pointed $0$-sphere, and $\CC\ast\CC$ is the product of two pointed $0$-spheres in the category of pointed noncommutative spaces $\iNS$.
 \end{ex}
 
 \section{Appendix} 
 
\emph{This appendix has been added after the publication of the article. It addresses a question that the author has received from several different sources.}
  
Let $\cC$ be a small pointed $\infty$-category with finite colimits and set $\cD = \Ind_\omega(\cC)$, i.e., the ind-completion of $\cC$ that demonstrates $\cD$ to be a presentable $\infty$-category. Let $\Sigma=\Sigma_\cD$ denote the suspension functor in $\cD$ and let $\Omega=\Omega_\cD$ denote its right adjoint. Recall that there is also an adjoint pair $\adj{\Sigma^\infty=\Sigma^\infty_\cD}{\cD}{\Sp(\cD)}{\Omega^\infty_\cD=\Omega^\infty}$.

\begin{lem}
 There is a natural identification $$\h\Sp(\cD)(\Sigma^\infty X,\Sigma^\infty Y) \cong \dlim_n \h\cC(\Sigma^n X,\Sigma^n Y)$$ for all $X,Y\in\cC$.
\end{lem}

\begin{proof}
 Due to the existence of finite colimits in $\cC$ and the preservation of finite colimits by the Yoneda embedding $j:\cC\map \cD$ (see Proposition 5.3.5.14 of \cite{LurToposBook}) we conclude that if $X\in\cC$, then so does $\Sigma^n X$ for all $n\in\NN$. Now 
 \beqn
\h\Sp(\cD)(\Sigma^\infty X,\Sigma^\infty Y) &\cong& \h\cD( X, \Omega^\infty\Sigma^\infty Y) \text{ [adjunction of $\Sigma^\infty,\Omega^\infty$]}\\
   &\cong& \h\cD(X,{\dlim}_n \Omega^n\Sigma^n Y) \\
   &\cong& {\dlim}_n \h\cD(X,\Omega^n\Sigma^n Y) \text{ [since $X$ is a compact object in $\cD$]}\\
   &\cong& {\dlim}_n \h\cD(\Sigma^n X,\Sigma^n Y) \text{ [adjunction of $\Sigma^n,\Omega^n$]}\\
   &\cong& {\dlim}_n \h\cC(\Sigma^n X,\Sigma^n Y) \text{ [since $j:\cC\map\cD$ is fully faithful]}.
 \eeqn
\end{proof}

\begin{ex}
The above argument is implicitly used for $\cC = \iCsep^\op$ in Remark \ref{Susp}.
\end{ex}

Another way to see the same is as follows: Let $\mathtt{{Cat}^{fincolim}_{\infty,\ast}}$ denote the $\infty$-category of small, pointed and finitely cocomplete $\infty$-categories with finite colimit preserving functors between them. Let $\cC[\Sigma^{-1}]$ be the direct limit in $\mathtt{{Cat}^{fincolim}_{\infty,\ast}}$ of the following diagram: $$\cC\overset{\Sigma}{\map}\cC\overset{\Sigma}{\map}\cC\overset{\Sigma}{\map} \cdots$$ Then one can show that $\Ind_\omega(\cC[\Sigma^{-1}])\simeq\Sp(\cD)$ using the fact that they both satisfy the same set of infinity categorical universal properties. It follows that there is a fully faithful functor $\h\cC[\Sigma^{-1}]\hookrightarrow\h\Ind_\omega(\cC[\Sigma^{-1}])\simeq\h\Sp(\cD)$. Now $\h\cC[\Sigma^{-1}]$ is equivalent to the Spanier--Whitehead category of $\h\cC$ with respect to the endofunctor $\Sigma$ and the morphisms in the Spanier--Whitehead category are precisely $\dlim_n \h\cC(\Sigma^n X,\Sigma^n Y)$ for all $X,Y\in\cC$ (for further details consult the appendix of \cite{BarJoaMe}).

%----------------------------------------bibliography----------------------------------------------------

\bibliographystyle{abbrv}

\bibliography{/home/ibatu/Professional/math/MasterBib/bibliography}

\end{document}